\documentclass[11pt]{amsart}
\usepackage{amssymb,amstext,amsmath,amscd,amsthm,amsfonts,enumerate,graphicx,latexsym,stmaryrd}
\usepackage[usenames]{color}
\usepackage[all]{xy}

\newtheorem{thm}{Theorem}[section]
\newtheorem{lem}[thm]{Lemma}

\newtheorem{cor}[thm]{Corollary}
\theoremstyle{definition}

\newtheorem{conj}[thm]{Conjecture}
\newtheorem{claim}{Claim}
\newtheorem*{claim*}{Claim}
\theoremstyle{remark}

\numberwithin{equation}{thm}

\def\cok{\operatorname{Coker}}
\def\cmz{\operatorname{\mathsf{CM}_0}}
\def\scmz{\operatorname{\underline{\mathsf{CM}}_0}}

\def\End{\operatorname{End}}
\def\Ext{\operatorname{Ext}}
\def\Hom{\operatorname{Hom}}
\def\im{\operatorname{Im}}
\def\ker{\operatorname{Ker}}
\def\lend{\operatorname{\underline{End}}}
\def\lhom{\operatorname{\underline{Hom}}}
\def\lmod{\operatorname{\mathsf{\underline{mod}}}}

\def\m{\mathfrak{m}}

\def\p{\mathfrak{p}}

\def\rank{\operatorname{rank}}

\def\syz{\mathsf{\Omega}}

\begin{document}
\allowdisplaybreaks
\title{The Huneke--Wiegand conjecture and middle terms of almost split sequences}
\author{Toshinori Kobayashi}
\address{Graduate School of Mathematics, Nagoya University, Furocho, Chikusaku, Nagoya, Aichi 464-8602, Japan}
\email{m16021z@math.nagoya-u.ac.jp}
\thanks{The author was partly supported JSPS Grant-in-Aid for JSPS Fellows 18J20660.}
\begin{abstract}
Let $R$ be a Gorenstein local domain of dimension one.
We show that a nonfree maximal Cohen--Macaulay $R$-module $M$ possessing more than one nonfree indecomposable summand in the middle term of the almost split sequence ending in $M$ has a nonvanishing self extension.
In other words, we show that the Huneke--Wiegand conjecture is affirmative for such $R$-modules $M$.
\end{abstract}
\maketitle
\section{Introduction}

In this paper, we study the following conjecture of Huneke and Wiegand; see \cite[the discussion following the proof of 5.2]{HW}.

\begin{conj}[Huneke and Wiegand \cite{HW}] \label{conj}
	Let $R$ be a Gorenstein local domain of dimension one.
	Let $M$ be a maximal Cohen--Macaulay $R$-module.
	If $M\otimes_R \Hom_R(M,R)$ is torsion-free, then $M$ is free.
\end{conj}

Huneke and Wiegand \cite{HW} showed that this conjecture is true for hypersurfaces.
Many other partial answers are known \cite{C,CGTT,GL,GTTT,HIW,R3}, but, the conjecture is still open in general.
Let $R$ be a Gorenstein local domain of dimension one.
A finitely generated $R$-module is torsion-free if and only if it is reflexive if and only if it is maximal Cohen--Macaulay.
Therefore Conjecture \ref{conj} implies the Auslander--Reiten conjecture for Gorenstein local domains (\cite[Proposition 5.10]{CT}).
Assume that $M$ is a torsion-free $R$-module.
Then it is remarkable that the torsion-freeness of $M\otimes_R \Hom_R(M,R)$ is equivalent to saying that $\Ext^1_R(M,M)$ is zero; see \cite[Theorem 5.9]{HJ}.

The main result of this paper is the following.

\begin{thm} \label{thm}
	Let $(R,\m)$ be a Gorenstein local domain of dimension one.
	Let $M$ be a nonfree indecomposable torsion-free $R$-module.
	Assume that the number of indecomposable summand in the middle term of the Auslander--Reiten sequence ending in $M$ is greater than one.
	Then one has $\Ext^1_R(M,M)\not=0$.
	Hence, Conjecture \ref{conj} holds true for $M$.
\end{thm}

Remark that Roy \cite{R3} showed that for one-dimensional graded complete intersections $R$ satisfying some condition on the a-invariant, the assertion of Theorem \ref{thm} holds.
Our result is local (not graded), and we do not assume that the ring is a complete intersection.

In section 2, we give some preliminaries.
In section 3, the proof of Theorem \ref{thm} is given.

\section{Irreducible homomorphisms and almost split sequences}
In this section, we prove lemmas needed to prove the main theorem.
In the rest of this paper, let $(R,\m)$ be a commutative Gorenstein henselian local ring, and all modules are finitely generated, unless otherwise stated.
We denote by $\cmz(R)$ the category of maximal Cohen--Macaulay $R$-modules $M$ such that $M_\p$ is $R_\p$-free for any nonmaximal prime ideal $\p$ of $R$.
For an $R$-module $M$, $\syz M$ (resp. $\syz^i M$) denotes the first (resp. $i$-th) syzygy module in the minimal free resolution of $M$.

For $R$-modules $M$ and $N$, let $\lhom_R(M,N)$ denote the quotient of $\Hom_R(M,N)$ by the set of homomorphisms from $M$ to $N$ factoring through a free $R$-module.
Since $R$ is Gorenstein, the stable category $\scmz(R)$ of $\cmz(R)$ is a triangulated category.
Its morphism set is equal to the stable homset $\lhom_R(-,-)$ and its shift functor is the functor taking $\syz$; see \cite[Chapter 1]{Ha} for instance.
Hence we obtain the following lemma.

\begin{lem} \label{l31}
	Let $M,N$ be $R$-modules in $\cmz(R)$.
	Then we have the following isomorphisms.\\
{\rm(1)} $\lhom_R(\syz M,N)\cong \Ext^1_R(M,N)$, \qquad
{\rm(2)} $\Ext^1_R(M,N)\cong \Ext^1_R(\syz M,\syz N)$, \\
{\rm(3)} $\lhom_R(M,N)\cong \lhom_R(\syz M,\syz N)$.
\end{lem}

On the set $\lhom_R(M,N)$, we also use the following lemma.

\begin{lem} \label{32}
	Let $M,N$ be $R$-modules having no free summands and $f\colon M\to N$ be a homomorphism factoring through a free $R$-module.
	Then the image $\im f$ of $f$ is contained in $\m N$.
\end{lem}

\begin{proof}
	Write $f=hg$ where $g\colon M \to F$ and $h\colon F\to N$ are homomorphisms with a free $R$-module $F$.
	Since $M$ has no free summands, $\im g$ is contained in $\m F$.
	Hence $\im f\subseteq h(\m F)\subseteq \m N$.
\end{proof}

Recall that a homomorphism $f\colon X \to Y$ of $R$-modules is said to be \textit{irreducible} if
it is neither a split monomorphism nor a split epimorphism, and for any pair of morphisms $g$ and $h$ such that $f=gh$, either $g$ is a split epimorphism or $h$ a split monomorphism.

\begin{lem} \label{ll}
	Let $M,N$ be $R$-modules having no free summands and $f,g\colon M\to N$ be homomorphisms.
	Assume that $g$ factors through a free $R$-module.
	Then 
\begin{enumerate}[\rm(1)]
	\item $f$ is an isomorphism if and only if so is $f+g$.
	\item $f$ is a split epimorphism if and only if so is $f+g$.
	\item $f$ is a split monomorphism if and only if so is $f+g$.
	\item $f$ is irreducible if and only if so is $f+g$.
\end{enumerate} 
\end{lem}

\begin{proof}
We only need to show one direction; we can view $f$ as $(f+g)-g$.

(1): Assume that $f$ is an isomorphism with an inverse homomorphism $h\colon N\to M$.
Then the composite homomorphisms $gh$ factor through some free $R$-modules.
It follows from Lemma \ref{32} that there are inclusions $\im gh\subseteq \m M$.
By Nakayama's lemma, we see that $(f+g)h$ is a surjective endomorphism of $M$, and hence are automorphisms.
Since $h$ is an isomorphism, it follows that $f+g$ is an isomorphism.

(2): Assume that there exists a homomorphism $s\colon N \to M$ such that $fs=\text{id}_N$.
We may apply (1) to the homomorphism $fs+gs$ to see that $(f+g)s$ is also an isomorphism.
This means that $f+g$ is a split epimorphism.
The item (3) can be checked in the same way.

(4): Assume that $f$ is irreducible.
According to the previous part, $f+g$ is neither a split monomorphism nor a split epimorphism.
By the assumption, $g$ is a composite $ba$ of homomorphisms $a\colon M\to F$ and $b\colon F\to N$ with a free $R$-module $F$.
If there is a factorization $f+g=dc$ for some homomorphisms $c\colon M\to X$ and $d\colon X\to N$, then they induce a decomposition $M \xrightarrow{{}^t[a,c]} F\oplus X \xrightarrow{[-b,d]} N$ of $f$.
By the irreducibility of $f$, either ${}^t[a,c]$ is a split monomorphism or $[-b,d]$ is a split epimorphism.
In the former case, we can take a homomorphism $[p,q]\colon F\oplus X \to N$ such that the composite $pa+qc=[p,q]\circ{}^t[a,c]$ is equal to the identity map of $N$.
Using (1), $qc$ is also an isomorphism.
This yields that $c$ is a split monomorphism.
In the latter case, we can see that $d$ is a split epimorphism by similar arguments.
Thus we conclude that $f+g$ is an irreducible homomorphism.
\end{proof}

Let $M$ be a nonfree indecomposable module in $\cmz(R)$.
Then there exists an almost split sequence ending in $M$.
Namely, there is a nonsplit short exact sequence
\[
0 \to \tau M \xrightarrow[]{f} E_M \xrightarrow[]{g} M \to 0
\]
in $\cmz(R)$ such that $N$ is indecomposable and for any maximal Cohen--Macaulay $R$-module $L$ and a homomorphism $h\colon L \to M$ which is not a split epimorphism, $h$ factors through $g$; see \cite[Chapter 2,3]{Y} for details.
Note that an almost split sequence ending in $M$ is unique up to isomorphisms of short exact sequences.
In particular, for any nonfree indecomposable $R$-module $M$ in $\cmz(R)$, the $R$-modules $\tau M$ and $E_M$ are unique up to isomorphisms.

\begin{lem} \label{21}
Let $M$ be a nonfree indecomposable module in $\cmz(R)$.
Consider the almost split sequences
\[
0 \to \tau M \xrightarrow{f} E_M \xrightarrow{g} M \to 0, \quad 0 \to \tau(\syz M) \to E_{\syz M} \to \syz M \to 0
\]
ending in $M$ and $\syz M$.
Then $\syz (E_M)$ is isomorphic to $E_{\syz M}$ up to free summands.
\end{lem}

\begin{proof}
By horseshoe lemma, there exists a short exact sequence $s\colon 0 \to \syz(\tau M) \xrightarrow{f'} \syz E_M \oplus P \xrightarrow{g'} \syz M \to 0$ with some free $R$-module $P$.
Here, the class $\underline{g'}\in \lhom_R(\syz E_M,\syz M)$ of $g'$ coincides with the image $\syz(\underline{g})$ of the class $\underline{g}$ of $g$ under the isomorphism $\syz\colon \lhom_R(E_M,M)\to \lhom_R(\syz E_M,\syz M)$ in Lemma \ref{l31}.
We want to show that the sequence $s$ is an almost split sequence ending in $\syz M$.
By Lemma \ref{ll} (2), we see that $g'$ is a split epimorphism if and only if $\underline{g'}h=\text{id}$ for some $h$ in the category $\scmz(R)$.
In view of the equivalence $\syz\colon \scmz(R)\to \scmz(R)$, $g'$ as well as $g$ is not a split surjection.
This means that $s$ is not a split exact sequence.

We fix a homomorphism $h'\colon X\to \syz M$ which is not a split epimorphism.
We can use the equivalence $\syz\colon \scmz(R)\to \scmz(R)$ again to obtain an equality $h'=g'p+rq$ with some homomorphism $p\colon X \to \syz E_M$, $q\colon X\to F$, $r\colon F\to \syz M$, where $F$ is a free module.
As $g'$ is an epimorphism and $F$ is free, $r$ factors through $g$.
This shows that $h'=g't$ for some $t\colon X\to \syz E_M$.
Consequently, $s$ is an almost split sequence ending in $\syz M$.
\end{proof}

Consider the almost split sequence
	\[
	0 \to \tau(M) \to E_M \to M \to 0
	\]
ended in $M$.
We define a number $\alpha(M)$ to be the number of nonfree indecomposable summand of $E_M$.

\begin{lem} \label{22}
Let $M$ be a nonfree indecomposable module in $\cmz(R)$.
Then $\alpha(M)=\alpha(\syz^i M)$ for all $i\ge 0$.
\end{lem}

\begin{proof}
This is a direct consequence of Lemma \ref{21}.
\end{proof}

The following two lemmas play key roles in the next section.
See \cite[Lemma 4.1.8]{R2} for details of the lemma below.

\begin{lem} \label{l4}
	Let $f\colon M\to N$ be an irreducible homomorphism such that $M$ and $N$ are indecomposable in $\cmz(R)$.
	Assume that $\dim R= 1$.
	Then $f$ is either injective or surjective.
\end{lem}

Recall that an $R$-module $M$ has \textit{constant rank} $n$ if one has an isomorphism $M_\p\cong R_\p^{\oplus n}$ for all associated primes $\p$ of $R$.

\begin{lem} \label{27}
	Let $M,N$ be nonfree indecomposable modules in $\cmz(R)$ having same constant rank.
	Let $f\colon M\to N$ be an irreducible monomorphism.
	Assume that $\dim R=1$.
	Then $\cok f$ is isomorphic to $R/\m$.
\end{lem}

\begin{proof}
	By the assumption that $f$ is an irreducible monomorphism, $f$ is not surjective.
	Hence we can take a maximal proper submodule $X$ of $N$ containing $\im f$.
	Remark that the quotient $N/X$ is isomorphic to $R/\m$ and hence $X$ and $N$ has same constant rank.
	Since $\dim R=1$, $X$ is an $R$-module contained in $\cmz(R)$.
	Thus we have a factorization $M\to X \to N$ of $f$ in $\cmz(R)$.
	By the irreducibility of $f$, it follows that either $M \to X$ is a split monomorphism or $X \to N$ is a split epimorphism.
	As $X$ is proper submodule of $N$, the later case cannot occur.
	Therefore, we obtain a split monomorphism $g\colon M \to X$.
	Then, by the equalities $\rank M=\rank N=\rank X$, $g$ is an isomorphism.
	This implies the desired isomorphisms $\cok f\cong N/X\cong R/\m$.
\end{proof}

\section{Proof of the main theorem}

In this section, we give a proof of Theorem \ref{thm}.

\begin{proof}[Proof of Theorem \ref{thm}]
	Since $R$ is a Gorenstein local ring of dimension one, $\tau(N)$ is isomorphic to $\syz N$ for any nonfree indecomposable $R$-module $N$ in $\cmz(R)$.
	We assume that $M$ is a nonfree indecomposable $R$-module in $\cmz(R)$ satisfying $\Ext^1_R(M,M)=0$ and want to show that $\alpha(M)=1$.
	We see from Lemma \ref{l31} that the isomorphisms
	\[
\Ext^1_R(\syz^{i+1} M,\syz^{i+1} M)\cong \Ext^1_R(M,M)=0
	\]
	hold for all $i\ge 0$.
	If $E_{\syz^i M}$ has a free summand, then $\tau(\syz^i M)=\syz^{i+1} M$ has an irreducible homomorphism into $R$.
	Hence $\syz^{i+1} M$ is a direct summand of the maximal ideal $\m$.
	Since $R$ is a domain, this means that $\syz^{i+1} M$ is isomorphic to $\m$.
	It follows that $\Ext^1_R(\m,\m)$ is zero, and so $R$ should be regular.
	Therefore, we may assume that $E_{\syz^i M}$ has no free summands for all $i\ge 0$.
	By lemma \ref{22}, it is enough to show that $\alpha(\syz^i M)=1$ for some $i \ge 0$.
	Thus by replacing $M$ with $\syz^i M$, we may assume that $\rank M$ is minimal in the set $\{\rank \syz^i M\mid i\ge 0\}$.
	
	Decompose $E_M=E_1\oplus\cdots\oplus E_n$ as a direct sum of indecomposable modules and consider the almost split sequence
	\[
	0 \to \syz M \xrightarrow[]{f={}^t(f_1,\dots,f_n)} E_1\oplus\cdots\oplus E_n \xrightarrow[]{g=(g_1,\dots,g_n)} M \to 0
		\]
	ended in $M$, where $f_p\colon \syz M\to E_p$ and $g_p\colon E_p\to N$ are irreducible homomorphisms and $n=\alpha(M)$.
	Lemma \ref{l4} guarantees that each of $f_1,\dots,f_n$ and $g_1,\dots,g_n$ is either injective or surjective.
	
\begin{claim}
	There is a number $p$ such that $f_p$ is injective.
\end{claim}
	
\noindent \textit{Proof of Claim 1}. Suppose that all of the $f_1,\dots,f_n$ are surjective.
Then we get equalities $\im g=\sum_i\im g_i=\sum_i\im g_if_i$.
Since $\lhom_R(\syz M,M)=0$ (Lemma \ref{l31}), it follows from Lemma \ref{32} that $\im g_if_i\subseteq \m M$ for all $i=1,\dots,n$.
This yields that $\im g\subseteq \m M$, which contradicts to that $g$ is surjective. $\square$
\begin{claim}
	If there is a number $p$ such that $f_p$ is injective and $g_p$ is surjective,
	then $\alpha(M)=1$.
\end{claim}

\noindent \textit{Proof of Claim 2}. Suppose that $f_p$ is injective and $g_p$ is surjective.
Since $\lhom_R(\syz M,M)=0$, there is a free $R$-module $F$ and homomorphisms $a\colon \syz M \to F$ and $b\colon F\to M$ such that $g_pf_p=ba$.
Since $F$ is free and $g_p$ is surjective, we have a factorization $b=g_pc$ with some homomorphism $c\colon F\to E_p$.
So we get an equality $g_p(f_p-ca)=0$.
In particular, $f_p-ca$ factors through the kernel $\ker g_p$ of $g_p$, i.e. $f_p-ca=ed$ with a homomorphism $d\colon \syz M\to \ker g_p$ and the natural inclusion $e\colon \ker g_p \to E_p$.
By Lemma \ref{ll} (2), the homomorphism $f_p-ca\colon M\to E_p$ is also irreducible.
Hence either $e$ is a split epimorphism or $d$ is a split monomorphism.
In the former case, the equality $\ker g_p=E_p$ follows.
It means that the map $g_p$ is zero.
This is a contradiction to the irreducibility of $g_p$.
So it follows that $d$ is a split monomorphism.
Then one has $\rank \syz M \le \rank \ker g_p=\rank E_p - \rank M$.
This forces that $n=1$.
$\square$

By Claim 1, we already have an integer $p$ such that $f_p$ is a monomorphism.
If $g_p$ is surjective, then by Claim 2 it follows that $\alpha(M)=1$.
Therefore, we may suppose that $g_p$ is injective.
Then the inequalities $\rank \syz M \le \rank E_p \le\rank M$ hold.
By the minimality of $\rank M$, we have $\rank \syz M=\rank E_p=\rank M$.
In this case, we see isomorphisms $\cok f_p\cong R/\m\cong \cok g_p$ by Lemma \ref{27}.
Therefore, equalities $\ell(M/\im(f_pg_p))=\ell(\cok f_p)+\ell(\cok g_p)=2$ hold (here, $\ell(X)$ denotes the length for an $R$-module $X$).
By Lemma \ref{32}, $\im (f_pg_p)\subseteq \m M$.
So it follows that $\ell(M/\m M)\le 2$.
In other words, $M$ is generated by two elements as an $R$-module.
Since $M$ is nonfree, one has $\rank M=1$ and $\Hom_R(M,R)\cong \syz M$.
As $\rank \syz M=\rank M=1$, we can apply the same argument above for $\syz M$ to see that $\syz M$ is also generated by two elements.
Then by \cite[Theorem 3.2]{He}, one can see that $\Ext^1_R(M,M)\not=0$, a contradiction.
\end{proof}


\end{document}